\documentclass[11pt]{article}
\usepackage{indentfirst}
\usepackage{latexsym}
\usepackage{hyperref}
\usepackage{graphicx}
\usepackage{mathrsfs}
\usepackage{bookmark}
\usepackage{enumerate}
\usepackage{amsfonts,amsthm,amssymb,mathtools}
\usepackage{amsmath}
\usepackage{latexsym, euscript, epic, eepic, color}
\usepackage{multirow}
\usepackage{multicol}
\usepackage{amssymb}

\usepackage{enumerate}
\usepackage[all]{xy}
\usepackage{setspace}
\allowdisplaybreaks

\usepackage{epstopdf}
\numberwithin{equation}{section}
\newtheorem{theorem}{Theorem}[section]
\newtheorem{corollary}[theorem]{Corollary}
\newtheorem{definition}[theorem]{Definition}

\newtheorem{lemma}[theorem]{Lemma}

\newtheorem{claim}{Claim}
\newtheorem{remark}{Remark}

\allowdisplaybreaks

\begin{document}
	\baselineskip=16pt

	\title{Eigenvalues and spanning trees with constrained degree}
	
	\author{ Chang Liu, Jianping Li\thanks{Corresponding author: Jianping Li (lijanping65@nudt.edu.cn)  }
		\\
		\small \emph{College of Sciences, National University of Defense Technology,} \\
		\small  \emph{Changsha, China, 410073.}
	}

	\date{\today}
	
	\maketitle
	
	\begin{abstract}
		In this paper, we study some spanning trees with bounded degree and leaf degree from eigenvalues. For any integer $k\geq2$, a $k$-tree is a spanning tree in which every vertex has degree no more than $k$. Let $T$ be a spanning tree of a connected graph. The leaf degree of $T$ is the maximum number of end-vertices attached to $v$ in $T$ for any $v\in V(T)$. By referring to the technique shown in [Eigenvalues and $[a,b]$-factors in regular graphs, J. Graph Theory. 100 (2022) 458-469], for an $r$-regular graph $G$, we provide an upper bound for the fourth largest adjacency eigenvalue of $G$ to guarantee the existence of a $k$-tree. Moreover, for a $t$-connected graph, we prove a tight sufficient condition for the existence of a spanning tree with leaf degree at most $k$ in terms of spectral radius. This generalizes a result of Theorem 1.5 in [Spectral radius and spanning trees of graphs, Discrete Math. 346 (2023) 113400]. Finally, for a general graph $G$, we present two sufficient conditions for the existence of a spanning tree with leaf degree at most $k$ via the Laplacian eigenvalues of $G$ and the spectral radius of the complement of $G$, respectively.
		
	\end{abstract}

	

	\section{Introduction}
	In this paper, we deal only with simple, connected and undirected graphs. For a graph $G$ on $n$ vertices, we denote its vertex set by $V(G)=\{v_1,v_2,\dots,v_n\}$ and edge set by $E(G)$. The neighborhood of a vertex $v\in V(G)$ is the set $N_{G}(v)=\{u|uv\in E(G)\}$. The degree of $v\in V(G)$, denoted by $d_{G}(v)$, is defined as the number of neighbors of $v$ in $G$, i.e., $d_{G}(v)=|N_{G}(v)|$.  A vertex with degree $1$ is called an end-vertex. The leaf degree of a tree $T$ is the maximum number of end-vertices attached to $v$ in $T$ for any $v\in V(T)$. The adjacency matrix $A(G)$ of $G$ is the $n$-by-$n$ matrix in which entry $a_{ij}$ is 1 or 0 according to whether $v_i$ and $v_j$ are adjacent or not. The eigenvalues of $A(G)$, written $\lambda_1(G)\geq\lambda_2(G)\geq\dots\geq\lambda_n(G)$, is called the eigenvalues of $G$. Note that the largest eigenvalue $\lambda_1(G)$ equals the spectral radius of $G$. The Laplacian matrix of $G$ is defined as $L(G)=D(G)-A(G)$, where $D(G)$ is the diagonal matrix of vertex degrees of $G$. Let $\mu_1(G)\geq\mu_2(G)\geq\dots\geq\mu_n(G)=0$ be eigenvalues of $L(G)$. 
	
	Let $S$ be a nonempty subset of $V(G)$, then $G-S$ and $G[S]$ denote the subgraphs of $G$ induced by $V(G)\backslash S$ and $S$, respectively. The complement of $G$ is denoted by $\overline{G}$. For two vertex-disjoint subsets $S,T$ of $G$, the number of edges between $S$ and $T$ in $G$ is denoted by $e_{G}(S,T)$.  
	If $G_1$ and $G_2$ are two vertex-disjoint graphs, the union of $G_1$ and $G_2$, written $G_1\cup G_2$, is a graph whose vertex set $V(G_1\cup G_2)=V(G_1)\cup V(G_2)$ and edge set $E(G_1\cup G_2)=E(G_1)\cup E(G_2)$. The joint of $G_1$ and $G_2$, written $G_1\vee G_2$, is the graph obtained from the union $G_1\cup G_2$ by adding all possible edges between $G_1$ and $G_2$.

	The relationship between eigenvalues and the existence of various spanning subgraphs is a classical and hot topic, attracting lots of attention (see \cite{boll,duan,fan,gu1,kim1,kim2,lu1,lu2,o2} and the reference therein). In particular, every connected graph contains a spanning tree. As a result, we are interested in finding a spanning tree with degree (leaf degree) constraints, i.e., the degree (leaf degree) of each vertex is bounded from below and/or from above by some fixed number. For an integer $k\geq2$, a $k$-tree is a spanning tree in which every vertex has degree at most $k$. Clearly, a Hamiltonian path is nothing but a spanning $2$-tree. In 1989, Win \cite{win} proved a sufficient condition for a connected graph to contain a $k$-tree, where $k\geq3$ (also see \cite{ell}).
	
	\begin{theorem}\label{ktree}\cite{ell,win}
		Let $G$ be a connected graph and $k\geq3$ be an integer. If for any $S\subseteq V(G)$, $c(G-S)\leq (k-2)|S|+2$, then $G$ has a $k$-tree, where $c(G-S)$ denotes the number of components in $G-S$.
	\end{theorem}
	In 2021, Fan et al. \cite{fan} show some upper bounds for the
	adjacency spectral radius and the signless Laplacian spectral radius to guarantee the existence of a $k$-tree in a connected graph $G$. Later, Gu and Liu \cite{gu2} provided a sufficient condition for the existence of a $k$-tree via Laplacian eigenvalues.

	In 2001, Kaneko \cite{kan} gave a definition of leaf degree of a spanning tree and proved a necessary and sufficient condition for the existence of a spanning tree with leaf degree at most $k$ in connected graphs.
	\begin{theorem}\label{leafk}\cite{kan}
		Let $G$  be a connected graph and $k\geq1$ be an integer. Then $G$ has a spanning tree with leaf degree at most $k$ if and only if $i(G-S)<(k+1)|S|$ for every nonempty subset $S\subseteq V(G)$, where $i(G-S)$ is the number of isolated vertices of $G-S$. 
	\end{theorem}
	Ao et al. \cite{ao} gave a tight spectral condition for the existence of a spanning tree with leaf degree at most $k$ in a connected graph and determined extremal graphs, where $k\geq1$ is an integer.
	
	There is a long history of research on factors (some spanning subgraphs with constrained degree) from graph eigenvalues in regular graphs. If every vertex in a graph has the same degree of a positive integer $r$, the graph is said to be $r$-regular. Brouwer and Haemers \cite{brou} started to investigate the relations between the third largest eigenvalue and the existence of a perfect matching (1-factor) in regular graphs. Their results were improved in \cite{cioa} and extended in \cite{kim2,lu3} ($[1,b]$-odd factor). In 2022, for an $h$‐edge‐connected $r$‐regular graph $G$, O \cite{o1} presented some sufficient spectral conditions for the existence of an even $[a,b]$‐factor or an odd $[a,b]$‐factor. Very recently, the connections between the existence of a $(g,f)$-parity factor and certain eigenvalues in an $h$-edge-connected graph $G$ with given minimum degree have been considered by Kim and O \cite{kim1}. 
	
	The existence of spanning subgraphs is in close relationship with the connectivity of graphs. The vertex connectivity of $G$, also called ``connectivity" simply, is the minimum size of a vertex cut, i.e., the minimum number of vertices of which the deletion induces a disconnected graph or a single vertex. For a positive integer $t$, we say $G$ is a $t$-connected graph if its connectivity at least $t$. In 2020, O \cite{o2} started to investigate the relations between the spectral radius of a connected graph and the existence of a perfect matching. Kim et al. \cite{kim3} provided a sharp upper bound for the spectral radius in a connected graph with fixed order and bounded matching number. In 2022, Zhang \cite{zhang} generalized the results in \cite{kim3,o2}  among $t$-connected graphs, where $t\geq1$.

	Inspired by the works in \cite{ao,fan,gu2,kim1,kim2,o1,zhang}, it is natural to consider the existence of two spanning trees with constrained degree (leaf degree) of connected graph. By referring to the technique shown in  \cite{kim1,o1}, for a $r$-regular graph $G$, we provide an upper bound for the fourth largest adjacency eigenvalue of $G$ to guarantee the existence of a $k$-tree. For a $t$-connected graph, we prove a tight sufficient condition for the existence of a spanning tree with leaf degree at most $k$ in terms of spectral radius. This generalizes a result of Theorem 1.5 in \cite{ao}. Finally, for a general graph $G$, we present two sufficient conditions for the existence of a spanning tree with leaf degree at most $k$ via the Laplacian eigenvalues of $G$ and the spectral radius of the complement of $G$, respectively.

	\section{Preliminaries}
	
	In this section, we introduce some notions and lemmas, which are useful in the proof of our main results.
	\begin{lemma}\label{radius}\cite{brou2,gods}
		If $H$ is an induced subgraph of $G$, then $\lambda_{i}(H)\leq \lambda_{i}(G)$. In particular, if $H$ is proper, then $\lambda_{i}(H)<\lambda_{i}(G)$ for $1\leq i\leq |V(H)|$.
	\end{lemma}
	
	\begin{lemma}\cite{gods}\label{nonnegative}
		If $A$ is a real nonnegative matrix of order $n$, and $A'$ is a real nonnegative matrix of order $n$ such that $A-A'$ is nonnegative, then $\lambda_{1}(A')\leq\lambda_{1}(A)$.
	\end{lemma}

	
	\begin{definition}\label{equpart}\cite{brou2}
		Let $M$ be a real $n\times n$ matrix and let $\pi=\{P_1,P_2,\dots,P_t\}$ be a partition of $[n]=\{1,2,\dots,n\}$. Then the matrix $M$ can be described in the following block form
		\begin{equation*}
			M=\begin{bmatrix}
				M_{1,1} & M_{1,2} & \cdots & M_{1,t}\\
				M_{2,1} & M_{2,2} & \cdots & M_{2,t}\\
				\vdots & \vdots & \ddots & \vdots \\
				M_{t,1} & M_{t,2} & \cdots & M_{t,t}
			\end{bmatrix}.
		\end{equation*}
		The quotient matrix of $M$ with respect to the partition $\pi$ is the matrix $B_{\pi}=(b_{i,j})_{i,j=1}^{t}$ with
		\begin{equation*}
			b_{i,j}=\frac{1}{|P_i|}\boldsymbol{j}^{T}_{|P_i|}M_{i,j}\boldsymbol{j}_{|P_i|}
		\end{equation*} 
		for all $i,j\in\{1,2,\dots,t\}$, where $\boldsymbol{j}_n$ denotes the all ones vector in $\mathbb{R}^n$. The partition $\pi$ is called an equitable partition if each block $M_{i,j}$ of $M$ has constant row sum $b_{i,j}$. In this case, the quotient matrix $B_{\pi}$ is called an equitable quotient matrix of $M$.
	\end{definition}
	
	The largest eigenvalue of a matrix $M$ is denoted by $\rho(M)$.
	
	\begin{lemma}\label{equradius}\cite{brou2}
		Let $M$ be a real symmetric matrix, and $M$ is nonnegative and irreducible. Let $B$ be a quotient matrix of $M$ as defined in Definition \ref{equpart}. The eigenvalues of $B$ interlace the eigenvalues of $M$. If $B$ is an equitable quotient matrix of $M$,  then $\rho(M) = \rho(B)$.
	\end{lemma}

	\begin{lemma}[The Cauchy's interlace theorem \cite{hwang}]\label{cauchy}
		Let two sequences of real number, $\lambda_{1}\geq\lambda_{2}\geq\dots\geq\lambda_n$ and $\eta_1\geq\eta_2\geq\dots\geq\eta_{n-1}$, be the eigenvalues of symmetric matrices $A$ and $B$, respectively. If $B$ is a principal submatrix of $A$, then the
		eigenvalues of $B$ interlace the eigenvalues of $A$, i.e., $\lambda_{1}\geq\eta_1\geq\lambda_{2}\geq\dots\geq\eta_{n-2}\geq\lambda_{n-1}\geq\eta_{n-1}\geq\lambda_{n}$.
	\end{lemma}

	\section{The relationship between the fourth largest adjacency eigenvalues and $k$-trees of $r$-regular graphs}
	In this section, we will give some sufficient spectral conditions to guarantee
	the existence of $k$-trees in an $r$-regular graph. Firstly, we construct a special graph $H_{r,k}$. We will prove that $G$ contains no $k$-trees if $G$ contains $H_{r,k}$ as its subgraph. 
	
	\begin{definition}
		\begin{equation*}
			H_{r,k}=\begin{cases}
				K_{r-\lceil\frac{r}{k-2}\rceil+3 }\vee\overline{\frac{\lceil\frac{r}{k-2}\rceil-2}{2}K_2}, & \mbox{if } \lceil\frac{r}{k-2}\rceil\geq3 \mbox{ is even},\\
				K_{r-\lceil\frac{r}{k-2}\rceil+2 }\vee\overline{\frac{\lceil\frac{r}{k-2}\rceil-1}{2}K_2}, &	\mbox{if } \lceil\frac{r}{k-2}\rceil\geq3 \mbox{ is odd}, \\
				\left( K_1\cup K_2\right)\vee \overline{\frac{r-1}{2}K_2}, & \mbox{if } \lceil\frac{r}{k-2}\rceil=2 \mbox{ and } r \mbox{ is odd}.
			\end{cases}
		\end{equation*}
	\end{definition}

	The spectral radius of $H_{r,k}$ is denoted by $\rho_{r,k}$. By calculation, we get the following lemma.
	\begin{lemma}\label{lem_spec1}
		\begin{equation*}
			\rho_{r,k} = \begin{cases}
				\frac{r}{2}+\frac{\sqrt{r^2+4r-4\lceil\frac{r}{k-2}\rceil+12}}{2}-1, & \mbox{if } \lceil\frac{r}{k-2}\rceil\geq3 \mbox{ is even},\\
				\frac{r}{2}+\frac{\sqrt{r^2+4r-4\lceil\frac{r}{k-2}\rceil+8}}{2}-1, & \mbox{if } \lceil\frac{r}{k-2}\rceil\geq3 \mbox{ is odd},\\
				\theta(r), & \mbox{if } \lceil\frac{r}{k-2}\rceil=2 \mbox{ and } r \mbox{ is odd},
			\end{cases}
		\end{equation*}
	where $\theta(r)$ is the largest root of the equation $x^3 + (2 - r)x^2 - 2rx + r-1=0$.
	\end{lemma}
	\begin{proof}
		By an easy calculation, we can get some equitable quotient matrices of $A(H_{r,k})$, listed below. The equitable quotient matrix of $A\left( K_{r-\lceil\frac{r}{k-2}\rceil+3 }\vee\overline{\frac{\lceil\frac{r}{k-2}\rceil-2}{2}K_2} \right) $ with partition $\left\lbrace V(K_{r-\lceil\frac{r}{k-2}\rceil+3 }),\right. \\ \left. V\left( \overline{\frac{\lceil\frac{r}{k-2}\rceil-2}{2}K_2}\right) \right\rbrace$ equals
		\begin{equation*}
			B_1=\begin{bmatrix}
				r-\lceil\frac{r}{k-2}\rceil+2 & \lceil\frac{r}{k-2}\rceil-2 \\
				r-\lceil\frac{r}{k-2}\rceil+3 & \lceil\frac{r}{k-2}\rceil-4
			\end{bmatrix}.
		\end{equation*}
		
		The equitable quotient matrix of $A\left( K_{r-\lceil\frac{r}{k-2}\rceil+2 }\vee\overline{\frac{\lceil\frac{r}{k-2}\rceil-1}{2}K_2} \right) $ with partition $\left\lbrace V\left( K_{r-\lceil\frac{r}{k-2}\rceil+2 }\right),\right. \\ \left. V\left(\overline{\frac{\lceil\frac{r}{k-2}\rceil-1}{2}K_2} \right)  \right\rbrace $ is
		\begin{equation*}
			B_2=\begin{bmatrix}
				r-\lceil\frac{r}{k-2}\rceil+1 & \lceil\frac{r}{k-2}\rceil-1 \\
				r-\lceil\frac{r}{k-2}\rceil+2 & \lceil\frac{r}{k-2}\rceil-3
			\end{bmatrix}.
		\end{equation*}
		
		The equitable quotient matrix of $A\left( \left( K_1\cup K_2\right)\vee \overline{\frac{r-1}{2}K_2}\right) $ with partition $\left\lbrace V(K_1), V(K_2), V\left( \overline{\frac{r-1}{2}K_2}\right)  \right\rbrace $ has the form
		\begin{equation*}
			B_3=\begin{bmatrix}
				0 & 0 & r-1 \\
				0 & 1 & r-1\\
				1 & 2 & r-3
			\end{bmatrix},
		\end{equation*}
		and the characteristic polynomial of $B_3$ is $\kappa(x)=x^3 + (2 - r)x^2 - 2rx + r-1$.
	
	
		By Lemma \ref{equradius}, we obtain that
		\begin{equation*}
			\begin{split}
				\rho(B_1)&=\lambda_{1}\left( K_{r-\lceil\frac{r}{k-2}\rceil+3 }\vee\overline{\frac{\lceil\frac{r}{k-2}\rceil-2}{2}K_2} \right)=\frac{r}{2}+\frac{\sqrt{r^2+4r-4\lceil\frac{r}{k-2}\rceil+12}}{2}-1,\\
				\rho(B_2)&=\lambda_{1}\left( K_{r-\lceil\frac{r}{k-2}\rceil+2 }\vee\overline{\frac{\lceil\frac{r}{k-2}\rceil-1}{2}K_2} \right)=\frac{r}{2}+\frac{\sqrt{r^2+4r-4\lceil\frac{r}{k-2}\rceil+8}}{2}-1,\\
				\rho(B_3)&=\lambda_{1}\left( \left( K_1\cup K_2\right)\vee \overline{\frac{r-1}{2}K_2}\right) =\theta(r),
			\end{split}
		\end{equation*}
	where $\theta(r)$ is the largest root of equation $x^3 + (2 - r)x^2 - 2rx + r-1=0$. This completes the proof.
	\end{proof}

	
	
	\begin{theorem}\label{thmktree}
		Let $r\geq3$ and $3\leq k<r$ be two integers. Let $G$ be an $r$-regular graph on $n$ vertices. If  $n$ is sufficiently larger with compared to $r$ and $\lambda_{4}(G)$ is smaller than $\rho_{r,k}$, then $G$ has a $k$-tree.
	\end{theorem}
	\begin{proof}
		Assume to the contrary that $G$ has no $k$-tree. From Theorem \ref{ktree}, there exists a vertex set $S\subseteq V(G)$ such that $c(G-S)\geq (k-2)|S|+3$. Let $G_1,G_2,\dots,G_c$ be the components of $G-S$, where $c=c(G-S)$.
		
		\begin{claim}\label{claim1}
			There are at least four components, write $G_1,G_2,G_3,G_4 $, such that $e_{G}(V(G_i),S)<\lceil \frac{r}{k-2} \rceil$ for all $1\leq i\leq 4$.
		\end{claim}
		\begin{proof}[Proof of Claim \ref{claim1}]
			Assume to the contrary that there are at most three such components in $G-S$. Note that $G$ is $r$-regular. Then 
			\begin{equation*}
				\begin{split}
					r|S|&\geq\sum_{i=1}^{c}e_{G}(V(G_i),S)\\
					&\geq \lceil \frac{r}{k-2} \rceil (c-3)+3\\
					&\geq \lceil \frac{r}{k-2} \rceil \cdot (k-2) \cdot |S|+3\\
					&\geq r|S|+3,
				\end{split}
			\end{equation*}
		which is a contradiction. This completes the proof of Claim 1.
		\end{proof}
	By Lemma \ref{radius}, it follows that
	\begin{equation*}
		\lambda_{4}(G)\geq\lambda_{4}(G_1\cup G_2\cup G_3\cup G_4)\geq \min_{i\in\{1,2,3,4\}}\lambda_1(G_i).
	\end{equation*}

	\begin{claim}
		For an $r$-regular graph $G$, if $\lceil\frac{r}{k-2}\rceil=2$ and $r$ is even, then $G$ must have a $k$-tree.
	\end{claim}
	\begin{proof}[Proof of Claim 2]
		Since $G$ is an $r$-regular graph, we get $2|E(G_i)|=r|V(G_i)|-e_{G}(V(G_i),S)$ for any $1\leq i\leq c$.  Note that $r$ is even. Then $e_{G}(V(G_i),S)\geq2$ is even.  Recall that, there exists a vertex set $S$ such that $c(G-S)\geq(k-2)|S|+3$, it follows that
		\begin{equation*}
			\begin{split}
				r|S|&=2|E(G[S])|+\sum_{i=1}^{c}e_{G}(V(G_i),S)\\
				&\geq 2|E(G[S])|+2c\\
				&=2|E(G[S])|+\lceil\frac{r}{k-2}\rceil(c-3)+6\\
				&\geq 2|E(G[S])|+\lceil\frac{r}{k-2}\rceil(k-2)|S|+6\\
				&\geq2|E(G[S])|+r|S|+6,
			\end{split}
		\end{equation*}
	a contradiction. Hence $c(G-S)\leq(k-2)|S|+2$ for all $S\subseteq V(G)$, which implies that $G$ has a $k$-tree. This completes the proof of Claim 2.
	\end{proof}
	\begin{claim}\label{claim3}
		Let $\mathcal{H}$ be the set containing all components of $G-S$ satisfying $e_{G}(V(H),S)<\lceil\frac{r}{k-2}\rceil$, where $H\in \mathcal{H}$. Let
		\begin{equation*}
			\alpha=\begin{cases}
				\lceil \frac{r}{k-2} \rceil-2, & \mbox{if } \lceil \frac{r}{k-2} \rceil\geq3 \mbox{ is even},\\
				\lceil \frac{r}{k-2} \rceil-1, & \mbox{if } \lceil \frac{r}{k-2} \rceil\geq3 \mbox{ is odd}.
			\end{cases}
		\end{equation*}
		Then the following holds.
		\begin{itemize}
			\item [\textnormal{(\romannumeral1)}] For $\lceil \frac{r}{k-2} \rceil=2$ and $r$ is odd or $\lceil \frac{r}{k-2} \rceil\geq3$ , we have $\lambda_{1}(H)\geq\rho_{r,k}$.
			\item [\textnormal{(\romannumeral2)}] If $H_0\in\mathcal{H}$ and $\lambda_{1}(H_0)\leq\lambda_{1}(H)$ for any $H\in\mathcal{H}$, we have
		 \begin{equation*}
		 	|V(H_0)|=\begin{cases}
		 		r+1, & \mbox{if }	\lceil \frac{r}{k-2} \rceil\geq3,\\
		 		r+2, & \mbox{if } \lceil \frac{r}{k-2} \rceil=2  \mbox{ and } r \mbox{ is odd},
		 	\end{cases}
		 \end{equation*}
		 and
		 \begin{equation*}
		 \begin{cases}
		 			2|E(H_0)|\geq r(r+1)-\alpha, & \mbox{if }	\lceil \frac{r}{k-2} \rceil\geq3,\\
		 		    2|E(H_0)|=r(r+2)-1, & \mbox{if } \lceil \frac{r}{k-2} \rceil=2  \mbox{ and } r \mbox{ is odd}.
		 	\end{cases}
		 \end{equation*}
		\end{itemize}   
	\end{claim}
	\begin{proof}[Proof of Claim 3]
		Since $e_{G}(V(H),S)<\lceil \frac{r}{k-2} \rceil< r$ and $G$ is an $r$-regular graph, we have $|V(H)|\geq r+1$.  We distinguish it into two situations for consideration.
		
		\textbf{Case 1.} $\lceil \frac{r}{k-2} \rceil\geq3$.
		
		First, we suppose that $|V(H)|\geq r+2$. As is
		well known, we have $\lambda_{1}(H)\geq\frac{2|E(H)|}{|V(H)|}$ (see \cite{brou2}, for example). One can check that  
		\begin{equation*}
			\lambda_{1}(H)> \frac{r|V(H)|-\alpha}{|V(H)|}\geq r-\frac{\alpha}{r+2}.
		\end{equation*}
	
	Combining Lemma \ref{lem_spec1}, we have
	\begin{align*}
			r-\frac{\alpha}{r+2}>\rho_{r,k}&\Leftrightarrow	r-\frac{\alpha}{r+2}>\frac{r}{2}+\frac{\sqrt{r^2+4r-4\alpha+4}}{2}-1\\
			&\Leftrightarrow \frac{r+2}{2}>\frac{\sqrt{r^2+4r-4\alpha+4}}{2}+\frac{\alpha}{r+2}\\
			&\Leftrightarrow (r+2)^2-2\alpha>(r+2)\sqrt{r^2+4r-4\alpha+4}\\
			&\Leftrightarrow4\alpha^2>0.
	\end{align*}
	It is easy to see that $\lambda_1(H)>r-\frac{\alpha}{r+2}>\rho_{r,k}$ always holds for $\lceil \frac{r}{k-2} \rceil\geq3$ and $|V(H)|\geq r+2$. Recall that $\lambda_{1}(H_0)\leq\lambda_{1}(H)$ for any $H\in\mathcal{H}$ and $|V(H_{r,k})|=r+1$. Then $|V(H_0)|=r+1$. Clearly, $e_{G}(V(H_0),S)<\lceil \frac{r}{k-2} \rceil$ is even. We get $e_{G}(V(H_0),S)\leq \alpha$ and $2|E(H_0)|\geq r(r+1)-\alpha$. 

%
%
	
	Let $\beta\leq\alpha$ be an integer and $2|E(H_0)|=r(r+1)-\beta \geq r(r+1)-\alpha$. There are at least $r+1-\alpha$ vertices with degree $r$ in $H_0$ (Otherwise, $(r-\alpha)r+(r-1)(\alpha-1)=r(r+1)-\alpha-1<2|E(H_0)|$, a contradiction).	Let $V_1$ be a set of vertices with degree $r$ such that $|V_1| = r+1-\alpha$, and let $V_2$ be the remaining vertices in $V(H_0)$. The quotient matrix of the vertex partitions $\{V_1,V_2\}$ of $H_0$ is
	\begin{equation*}
		B_4=\begin{bmatrix}
			r-\alpha & \alpha \\
			r+1-\alpha & \alpha-1-\frac{\beta}{\alpha}
		\end{bmatrix}.
	\end{equation*}
	Let
	\begin{equation*}
		B'_4=\begin{bmatrix}
			r-\alpha & \alpha \\
			r+1-\alpha & \alpha-2
		\end{bmatrix}.
	\end{equation*}
	By Lemma \ref{nonnegative}, we have $\lambda_{1}(B_4)\geq\lambda_{1}(B'_4)$. By calculation, we get the largest eigenvalue of $B'_4$ is $\frac{r}{2}+\frac{\sqrt{r^2+4r-4\alpha+4}}{2}-1=\rho_{r,k}$.
	Combining with Lemma \ref{equradius}, we have $\lambda_{1}(H_0)\geq\lambda_{1}(B_4)\geq\lambda_{1}(B'_4)=\rho_{r,k}$ and hence $\lambda_{1}(H)\geq\rho_{r,k}$ for any $H\in\mathcal{H}$.
	
	\textbf{Case 2.} $\lceil \frac{r}{k-2} \rceil=2$ and $r$ is odd.
	
	Since $e_{G}(V(H),S)<\lceil \frac{r}{k-2} \rceil=2$ and $G$ is connected, we have $e_{G}(V(H),S)=1$. Notice that $2|E(H)|=r|V(H)|-e_{G}(V(H),S)$. Then $|V(H)|\geq r+2$ is an odd integer.
	If $|V(H)|\geq r+4$, then 
	\begin{equation*}
		\lambda_{1}(H)> \frac{r|V(H)|-\alpha}{|V(H)|}\geq r-\frac{1}{r+4}.
	\end{equation*}
	
	A straightforward calculation shows that $\psi'(x)=3x^2+2(2-r)x-2r$ and $\psi''(x)=2(3x-r+2)$.
	By plugging the value $r-\frac{1}{r+4}$ into $\psi(x)$, $\psi'(x)$ and $\psi''(x)$ yields that $\psi(r-\frac{1}{r+4})=\frac{r^3+6r^2-6r-57}{(r+4)^3}>0$ and $\psi'(r-\frac{1}{r+4})=\frac{r^4 + 10r^3 + 28r^2 + 12r - 13}{(r+4)^2}>0$ and $\psi''(r-\frac{1}{r+4})=\frac{4r^2+20r+10}{r+4}>0$ for $r\geq3$. Therefore, $\lambda_{1}(H)>r-\frac{1}{r+4}>\theta(r)$ holds for $|V(H)|\geq r+4$. Then we get $|V(H_0)|=r+2$ and $2|E(H_0)|=r(r+2)-1$. There exists $r+1$ vertices of degree $r$ in $H_0$ and one vertex of degree $r-1$ in $H_0$. One can see that $H_0\cong \left( K_1\cup K_2\right)\vee \overline{\frac{r-1}{2}K_2}$. The corresponding quotient matrix of $A\left( \left( K_1\cup K_2\right)\vee \overline{\frac{r-1}{2}K_2} \right) $ with the partition $\left\lbrace V(K_1), V(K_2), V\left( \overline{\frac{r-1}{2}K_2}\right)  \right\rbrace$ is
	\begin{equation*}
		B_5=\begin{bmatrix}
			0 & 0 & r-1 \\
			0 & 1 & r-1\\
			1 & 2 & r-3
		\end{bmatrix}.
	\end{equation*}
	The characteristic polynomial of $B_5$ is $\psi(x)=x^3 + (2 - r)x^2 - 2rx + r-1$ and the largest root of $\psi(x)=0$ is $\theta(r)=\rho_{r,k}$. Note that the partition is equitable, one can see that $\lambda_{1}(H_0)=\rho_{r,k}$, and $\lambda_{1}(H)\geq\rho_{r,k}$ for any $H\in\mathcal{H}$.
	
	This completes the proof of Claim 3.
	\end{proof}

	All in all, if $\lambda_{4}(G)<\rho_{r,k}$, we have $c(G-S)\leq (k-2)|S|+2$, implying $G$ has a $k$-tree. This completes the proof.
	\end{proof}
	
	\begin{remark}
		The condition that $n$ is sufficiently larger than $r$ implies that the graph $G$ is “sparse”. According to Claim \ref{claim1} and \ref{claim3}, if $G$ contains no $k$-trees, then there exists a vertex set $S\subseteq V(G)$ such that $c(G-S)\geq (k-2)|S|+3$, and furthermore, there are at least four components $G_i$ $(1\leq i\leq s)$ of $G-S$ with $|V(G_i)|\geq r+1$. This implies that $n$ at least $4(r+1)+|S|$. When $r$ is closer to $n$, indicating a denser graph $G$, the likelihood of $G$ containing a $k$-tree increases.	Moreover, the bound on $\lambda_4(G)$ given in Theorem \ref{thmktree} is not tight as Theorem \ref{ktree} does not present necessary and sufficient condition.
	\end{remark}

	\setcounter{claim}{0}

	\section{The relationship between the spectral radius and spanning trees with leaf degree at most $k$ of $t$-connected graphs}
	
	In this section, we present a sufficient spectral radius condition for the existence of a spanning tree with leaf degree at most $k$ in an $n$-vertex $t$-connected graph $G$. Let $i=i(G-S)$ and $s=|S|$. Firstly, we assume that $G$ has no spanning trees with leaf degree at most $k$, and then draw the following lemmas.
	
	\begin{lemma}\label{lem4_1} Let $G$ be an $n$-vertex graph without a spanning tree with leaf degree at most $k$. Then $\lambda_{1}(G)\leq \lambda_{1}\left( K_s\vee \left( K_{n-(k+2)s}\cup(k+1)sK_1\right)\right) $ with equality holds if and only if $G\cong K_s\vee \left( K_{n-(k+2)s}\cup(k+1)sK_1\right)$, where $1\leq s\leq \lfloor\frac{n}{k+2}\rfloor$.
	\end{lemma}
	\begin{proof}
		Since $G$ contains no spanning trees with leaf degree at most $k$, it follows from Theorem \ref{leafk} that there exists a vertex set $S\subseteq V(G)$ such that $i(G-S)\geq(k+1)|S|$, where $i(G-S)$ is the number of isolated vertices of $G-S$. Obviously, $G$ is a spanning subgraph of $K_s\vee \left( K_{n-(k+2)s}\cup(k+1)sK_1\right)$. By Lemma \ref{radius}, we get $\lambda_{1}(G)\leq \lambda_{1}\left( K_s\vee \left( K_{n-(k+2)s}\cup(k+1)sK_1\right)\right)$ with equality holds if and only if $G\cong K_s\vee \left( K_{n-(k+2)s}\cup(k+1)sK_1\right)$. This completes the proof.
	\end{proof}

	\setcounter{claim}{0}
	\begin{lemma}\label{lemma4_2}
		For $t\leq s\leq \lfloor\frac{n}{k+2}\rfloor$, 
		\begin{equation*}
			\begin{split}
				&\lambda_{1}(K_s\vee \left( K_{n-(k+2)s}\cup(k+1)sK_1\right))\\
				&\leq\max\left\lbrace \lambda_{1}(K_t\vee \left( K_{n-(k+2)t}\cup(k+1)tK_1\right)),\lambda_{1}\left( K_{\lfloor\frac{n}{k+2}\rfloor}\vee \left( K_{n-(k+2)\lfloor\frac{n}{k+2}\rfloor}\cup(k+1)\lfloor\frac{n}{k+2}\rfloor K_1\right)\right) \right\rbrace.
			\end{split}
		\end{equation*}
		  Equality holds if and only if $s=t$ or $s=\lfloor\frac{n}{k+2}\rfloor$.
	\end{lemma}
	\begin{proof}
		Clearly, the equitable quotient matrix of $A\left( K_s\vee \left( K_{n-(k+2)s}\cup(k+1)sK_1\right)\right)$ with partition $\left\lbrace V(K_s), V(K_{n-(k+2)s}), V((k+1)sK_1)\right\rbrace $ is equal to
		\begin{equation*}
			\begin{bmatrix}
				s-1 & n-(k+2)s & (k+1)s \\
				s & n-(k+2)s-1 & 0 \\
				s & 0 & 0
			\end{bmatrix}.
		\end{equation*}
	By calculation, we get the corresponding characteristic polynomial is
	\begin{equation*}
		f(x)=x^3+(s-n+ks+2)x^2+(s-n+ks-ks^2-s^2+1)x+ns^2-3ks^3-ks^2-s^2-2s^3-k^2s^3+kns^2.
	\end{equation*}
	The largest root of equation $f(x)=0$ is denoted by $\rho_f$. It follows from Lemma \ref{equradius} that $\lambda_{1}\left( K_s\vee \left( K_{n-(k+2)s}\cup(k+1)sK_1\right)\right)=\rho_f$. In the following, it suffices to show that
	\begin{equation*}
		\begin{split}
			&\lambda_{1}(K_s\vee \left( K_{n-(k+2)s}\cup(k+1)sK_1\right))\\
			&<\max\left\lbrace \lambda_{1}(K_t\vee \left( K_{n-(k+2)t}\cup(k+1)tK_1\right)),\lambda_{1}\left( K_{\lfloor\frac{n}{k+2}\rfloor}\vee \left( K_{n-(k+2)\lfloor\frac{n}{k+2}\rfloor}\cup(k+1)\lfloor\frac{n}{k+2}\rfloor K_1\right)\right) \right\rbrace
		\end{split}
	\end{equation*}
	for $t+1<s<\lfloor\frac{n}{k+2}\rfloor-1$. First, we consider the monotonicity of function $\rho_f=\rho_f(s)$ with respect to $s$. Define a function $\Phi(x,s)$ with two variables in the domain $D=\{(x,s)|x=\rho_f,1\leq s\leq \lfloor\frac{n}{k+2}\rfloor\}$ as follows:
	\begin{equation*}
		\Phi(x,s)=x^3+(s-n+ks+2)x^2+(s-n+ks-ks^2-s^2+1)x+ns^2-3ks^3-ks^2-s^2-2s^3-k^2s^3+kns^2.
	\end{equation*}
	
	We calculate that
	\begin{equation*}
		\frac{\partial \Phi(x,s)}{\partial x}=3x^2+(2s-2n+2ks+4)x+s-n+ks-ks^2-s^2+1\triangleq g_1(x).
	\end{equation*}
	Solve $g_1(x)=0$ yields that
	\begin{equation*}
		\begin{split}
			x_{1,1}&=\frac{n-s-ks-2}{3}+\frac{\sqrt{n^2-(2ks+2s+1)n+k^2s^2+5ks^2+ks+4s^2+s+1}}{3},\\
			x_{1,2}&=\frac{n-s-ks-2}{3}-\frac{\sqrt{n^2-(2ks+2s+1)n+k^2s^2+5ks^2+ks+4s^2+s+1}}{3}.
		\end{split}
	\end{equation*}
	Since $n\geq (k+2)s> ks+s+\frac{1}{2}$, we have $n^2-(2ks+2s+1)n+k^2s^2+5ks^2+ks+4s^2+s+1\geq(3k+4)s^2-s+1>0$ for $k\geq1$ and $s\geq1$. Hence $x_{1,1}>0$ is a real number. Moreover, one can check that $x_{1,2}<0$ for $n\geq(k+2)s$, $k\geq1$ and $s\geq1$. Then we calculate that
	\begin{equation*}
			\frac{\partial \Phi(x,s)}{\partial s}=(k+1)[x^2-(2s-1)x+(2ns-2s-3ks^2-6s^2)]\triangleq g_2(x).
	\end{equation*}	
	Solve $g_2(x)=0$ yields that
	\begin{equation*}
		\begin{split}
			x_{2,1}&=\frac{2s-1}{2}+\frac{\sqrt{4s-8ns+12ks^2+28s^2+1}}{2},\\
			x_{2,2}&=\frac{2s-1}{2}-\frac{\sqrt{4s-8ns+12ks^2+28s^2+1}}{2}.
		\end{split}
	\end{equation*}
	If $n>(\frac{3}{2}k+\frac{7}{2})s+\frac{1}{2}+\frac{1}{8s}$, the equation $g_2(x)=0$ has no real roots, which implies that $\frac{\partial \Phi(x,s)}{\partial s}>0$. For $n<\left( \frac{3}{2}k+3\right)s+1 $, we calculate that $x_{2,2}<0$.

	\begin{claim}
		$g_1(\rho_f)=3(\rho_f)^2+(2s-2n+2ks+4)\rho_f+s-n+ks-ks^2-s^2+1>0$ for $n\geq (k+2)s$.
	\end{claim}
	\begin{proof}[Proof of Claim 1]
		By plugging the value $x_{1,1}$ into $x$ of $\Phi(x,s)$, we get
		\begin{align*}
				\Phi(x_{1,1},s)&=\frac{1}{9}\left[ (2s+2ks+1)n^2-\frac{2}{3}n^3+(-2k^2s^2+2ks^2-2ks+4s^2-2s+1)n\right.\\ &\left.+\left(\frac{2k^3}{3}-4k^2-19k-\frac{43}{3}\right)s^3+(k^2-k-2)s^2+(-k-1)s-\frac{2}{3}\right]\\
				&+\frac{1}{9}\left[ \left( \frac{4s}{3}+\frac{4ks}{3}+\frac{2}{3}\right)n-\frac{2}{3}n^2-\frac{2k^2s^2}{3}-\frac{10ks^2}{3}-\frac{2ks}{3}-\frac{8s^2}{3}-\frac{2s}{3}-\frac{2}{3} \right]\\
				&\cdot \sqrt{n^2-(2ks+2s+1)n+k^2s^2+5ks^2+ks+4s^2+s+1}\\
				&\triangleq h_1(n).
		\end{align*}
	
	Taking the derivative of $h_1(n)$, we get
	\begin{align*}
			h'_1(n)&=\frac{1}{9\sqrt{n^2-(2ks+2s+1)n+k^2s^2+5ks^2+ks+4s^2+s+1}}\\
			&\cdot\left\lbrace \left[ -2n^3+(6ks+6s+3)n^2+(-6k^2s^2-18ks^2-6ks-12s^2-6s-3)n+2k^3s^3+12k^2s^3\right.\right.
			\\ &\left.\left.+3k^2s^2+18ks^3+9ks^2+3ks+8s^3+6s^2+3s+1 \right]+\left[ -2n^2+(4s+4ks+2)n-2k^2s^2\right.\right.\\
			&\left.\left.+2ks^2-2ks+4s^2-2s+1\right]\cdot \sqrt{n^2-(2ks+2s+1)n+k^2s^2+5ks^2+ks+4s^2+s+1} \right\rbrace.
	\end{align*}
	
	Let $\phi_1(n)=-2n^3+(6ks+6s+3)n^2+(-6k^2s^2-18ks^2-6ks-12s^2-6s-3)n+2k^3s^3+12k^2s^3+3k^2s^2+18ks^3+9ks^2+3ks+8s^3+6s^2+3s+1$ and $\phi_2(n)=-2n^2+(4s+4ks+2)n-2k^2s^2+2ks^2-2ks+4s^2-2s+1$. Thus,
	\begin{equation*}
		h'_1(n)=\frac{\phi_1(n)}{9\sqrt{n^2-(2ks+2s+1)n+k^2s^2+5ks^2+ks+4s^2+s+1}}+\frac{\phi_2(n)}{9}.
	\end{equation*}
	
	Taking the derivative of $\phi_1(n)$, we get
	\begin{equation*}
		\phi'_1(n)=-6n^2+(12ks+12s+6)n-(6k^2+18k+12)s^2-(6k+6)s-3.
	\end{equation*}
	It is easy to see that $\phi_1'(n)\leq \phi'_1(ks+s+\frac{1}{2})=-6ks^2-6s^2-\frac{3}{2}<0$. Since $n\geq (k+2)s$, we have $\phi_1(n)\leq\phi_1((k+2)s)=(-6k-8)s^3+(3k+6)s^2-3s+1<0$ for $k\geq1$ and $s\geq1$. Furthermore, we obtain that $\phi_2(n)\geq0$ for $n\in\left[(k+2)s,ks+s+\frac{\sqrt{3(4ks^2+4s^2+1)}}{2}+\frac{1}{2}\right]$ and $\phi_2(n)<0$ for $n>ks+s+\frac{\sqrt{3(4ks^2+4s^2+1)}}{2}+\frac{1}{2}$. Let $n\in\left[(k+2)s,ks+s+\frac{\sqrt{3(4ks^2+4s^2+1)}}{2}+\frac{1}{2}\right]$, we consider the following inequality:	
	\begin{align}
			&h'_1(n)=\frac{\phi_1(n)}{9\sqrt{n^2-(2ks+2s+1)n+k^2s^2+5ks^2+ks+4s^2+s+1}}+\frac{\phi_2(n)}{9}\geq0\nonumber\\
			&\Leftrightarrow \phi_1(n)+\phi_2(n)\sqrt{n^2-(2ks+2s+1)n+k^2s^2+5ks^2+ks+4s^2+s+1}\geq0\nonumber\\
			&\Leftrightarrow \phi_2(n)\sqrt{n^2-(2ks+2s+1)n+k^2s^2+5ks^2+ks+4s^2+s+1}\geq-\phi_1(n)\nonumber\\
			&\Leftrightarrow -9\cdot(4ks^2+4s^2+1)\cdot\left[ n^2-(2ks+2s+1)n+k^2s^2+5ks^2+ks+4s^2+s+1\right]\nonumber\\
			&\cdot \left[ n-\left( s+ks+\frac{1}{2}-\frac{\sqrt{4ks^2+4s^2+1}}{2}\right) \right]\cdot \left[ n-\left( s+ks+\frac{1}{2}+\frac{\sqrt{4ks^2+4s^2+1}}{2}\right) \right]\geq 0\nonumber\\
			&\Leftrightarrow \left[ n-\left( s+ks+\frac{1}{2}-\frac{\sqrt{4ks^2+4s^2+1}}{2}\right) \right]\cdot \left[ n-\left( s+ks+\frac{1}{2}+\frac{\sqrt{4ks^2+4s^2+1}}{2}\right) \right]\leq 0.\label{eq4.1}
	\end{align}

	Combining with \eqref{eq4.1}, we get
	\begin{equation*}
		\begin{cases}
		h'_1(n)\geq 0, & \mbox{if } n\in\left[(k+2)s, s+ks+\frac{1}{2}+\frac{\sqrt{4ks^2+4s^2+1}}{2}\right],\\
		h'_1(n)< 0, & \mbox{if } n> s+ks+\frac{1}{2}+\frac{\sqrt{4ks^2+4s^2+1}}{2}.
		\end{cases}
	\end{equation*}

	Therefore,
	\begin{equation*}
		\Phi(x_{1,1},s)\leq 	h_1\left( s+ks+\frac{1}{2}+\frac{\sqrt{4ks^2+4s^2+1}}{2}\right) =-9ks^3-9s^3<0,
	\end{equation*}
	which implies that $x_{1,1}<\rho_f$ and $\frac{\partial \Phi(x,s)}{\partial x}>0$ for $x=\rho_f$ and $n\geq (k+2)s$. This completes the proof of Claim 1.
%
%
	\end{proof}

	\begin{claim}
		For $(k+2)s\leq n\leq (\frac{3}{2}k+\frac{7}{2})s+\frac{1}{2}+\frac{1}{8s}$, there exists a real number $n_0\in\left[(k+\frac{11}{4})s+\frac{1}{2},\right.\\ \left. ks+2s+\frac{\sqrt{4ks^2+8s^2+1}}{2}+\frac{1}{2} \right] $ such that
		\begin{equation*}
			\begin{cases}
				g_2(x)<0, & \mbox{if } x=\rho_f \mbox{ and } n\in \left[(k+2)s, n_0 \right), \\
				g_2(x)\geq0, & \mbox{if } x=\rho_f \mbox{ and } n\in \left[n_0,  (\frac{3}{2}k+\frac{7}{2})s+\frac{1}{2}+\frac{1}{8s}\right].
			\end{cases}
		\end{equation*} 
	\end{claim}
	\begin{proof}[Proof of Claim 2]
		By plugging the value $x_{2,1}=\frac{2s-1}{2}+\frac{\sqrt{4s-8ns+12ks^2+28s^2+1}}{2}$ into $x$ of $\Phi(x,s)$ yields that
		\begin{equation*}
			\begin{split}
				\Phi(x_{2,1},s)&=s\left[ 2n^2+(-15s-4ks-2)n+2k^2s^2+16ks^2+2ks+27s^2+\frac{15}{2}s+1\right]\\&+s\left( -2n+\frac{11}{2}s+2ks+1\right)\sqrt{4s-8ns+12ks^2+28s^2+1}\\
				&\triangleq s\cdot h_2(n).
			\end{split}
		\end{equation*}
	Let $\varphi_1(n)=2n^2+(-15s-4ks-2)n+2k^2s^2+16ks^2+2ks+27s^2+\frac{15}{2}s+1$ and $\varphi_2(n)=-2n+\frac{11}{2}s+2ks+1$. Then we obtain	
	\begin{equation*}
		h_2(n)=\varphi_1(n)+\varphi_2(n)\sqrt{4s-8ns+12ks^2+28s^2+1}.
	\end{equation*}
 	It can be checked that $\varphi_1(n)\geq \varphi_1\left( (\frac{3}{2}k+\frac{7}{2})s+\frac{5}{8}\right)=\frac{1}{32}\left[ (16k^2 + 16k - 32)s^2 + (8k - 4)s + 17\right] >0$ for $k\geq1$ and $s\geq1$. In addition, we have $\varphi_2(n)\geq0\Leftrightarrow n\leq \left( k+\frac{11}{4}\right)s+\frac{1}{2} $. Hence for $n\in \left[ (k+2)s, \left( k+\frac{11}{4}\right)s+\frac{1}{2} \right] $, $\Phi(x_{2,1},s)>0$.
 	
 	The Cauchy's interlace theorem (see Lemma \ref{cauchy}) says that the second largest root of $f(x)$, denoted by $\xi_f$, is no more than $n-(k+2)s-1$, i.e., $\xi_f\leq n-(k+2)s-1$. Then we consider the following inequality:
 	\begin{equation*}
 		\begin{split}
 			&x_{2,1}=\frac{2s-1}{2}+\frac{\sqrt{4s-8ns+12ks^2+28s^2+1}}{2}\geq n-(k+2)s-1\\
 			&\Leftrightarrow \sqrt{4s-8ns+12ks^2+28s^2+1}\geq 2[n-(k+3)s-\frac{1}{2}]\\
 			&\Leftrightarrow -n^2+(4s+2ks+1)n-k^2s^2-3ks^2-ks-2s^2-2s\geq 0\\
 			&\Leftrightarrow ks+2s-\frac{\sqrt{4ks^2+8s^2+1}}{2}+\frac{1}{2}\leq n\leq ks+2s+\frac{\sqrt{4ks^2+8s^2+1}}{2}+\frac{1}{2}.
 		\end{split}
 	\end{equation*}
 	
 	Easily, we get
 	\begin{equation*}
 		ks+2s-\frac{\sqrt{4ks^2+8s^2+1}}{2}+\frac{1}{2}<(k+2)s<\left( k+\frac{11}{4}\right) s+\frac{1}{2}<ks+2s+\frac{\sqrt{4ks^2+8s^2+1}}{2}+\frac{1}{2}.
 	\end{equation*}
 	Thus, $x_{2,1}>n-(k+2)s-1\geq \xi_f$ for $n\in\left[ (k+2)s, \left( k+\frac{11}{4}\right)s+\frac{1}{2} \right] $, implying $x_{2,1}>\rho_f$.
 	
 	For $n\in \left(\left( k+\frac{11}{4}\right)s+\frac{1}{2}, (\frac{3}{2}k+\frac{7}{2})s+\frac{1}{2}+\frac{1}{8s}\right]$, let
 	\begin{equation*}
 		\begin{split}
 			h_3(n)&=(\varphi_1(n))^2-(-\varphi_2(n)\sqrt{4s-8ns+12ks^2+28s^2+1})^2\\
 			&=4n^4+(-28s-16ks-8)n^3+[(24k^2+72k+45)s^2+(24k+42)s+4]n^2+\left[ \left( -16k^3\right.\right.\\&
 			\left.\left.-60k^2-32k+48\right)s^3+(-24k^2-72k-45)s^2+(-8k-14)s \right]n+ \left( 4k^4+16k^3-12k^2\right.\\
 			&\left.-115k-118\right) s^4+(8k^3+30k^2+16k-24)s^3 +(4k^2+12k+8)s.
 		\end{split}
 	\end{equation*}
 
 	An easy calculation shows that
 	\begin{equation*}
 		\begin{split}
 			&h'_3(n)=16n^3+(-84s-48ks-24)n^2+(48k^2s^2+144ks^2+48ks+90s^2+84s+8)n\\
 			&+(-16k^3-60k^2-32k+48)s^3+(-24k^2-72k-45)s^2+(-8k-14)s,
 		\end{split}
 	\end{equation*}
 	and
 	\begin{equation*}
 			h''_3(n)=2\left[ 24n^2+(-84s-48ks-24)n+24k^2s^2+72ks^2+24ks+45s^2+42s+4\right].
 	\end{equation*}
 	
 	If the equation $h_3(n)=0$ has four real roots, we will write them as $r_1$, $r_2$, $r_3$, and $r_4$ in non-increasing order. Observe the following inequality with respect to: 	
 	\begin{align*}
 		&ks+2s+\frac{\sqrt{4ks^2+8s^2+1}}{2}+\frac{1}{2}<\left( \frac{3}{2}k+\frac{7}{2}\right)s+\frac{1}{2}\\
 		&\Leftrightarrow \sqrt{4ks^2+8s^2+1}<ks+3s\\
 		&\Leftrightarrow -(k^2+2k+1)s^2+1<0,
 	\end{align*}
 	we have $ks+2s+\frac{\sqrt{4ks^2+8s^2+1}}{2}+\frac{1}{2}<\left( \frac{3}{2}k+\frac{7}{2}\right)s+\frac{1}{2}$ always holds.  Therefore, for $k\geq1$ and $s\geq1$, we calculate the following three systems (see Fig \ref{fig}):
 	\begin{itemize}
 		\item  \begin{equation*}
 			\begin{cases}
 				h_3\left( \left( k+\frac{11}{4}\right)s+\frac{1}{2}\right)=\left( k^2+\frac{7}{4}k+\frac{49}{64} \right)s^4+\left( k+\frac{7}{8}\right)s^2+\frac{1}{4}>0,\\
 				h'_3\left( \left( k+\frac{11}{4}\right)s+\frac{1}{2}\right)=-(8k+7)s^3-4s<0,\\
 				h''_3\left( \left( k+\frac{11}{4}\right)s+\frac{1}{2}\right)=-24ks^2-9s^2-4<0.
 			\end{cases}
 		\end{equation*}
 	\item \begin{equation*}
 		\begin{split}
 			&h_3\left( ks+2s+\frac{\sqrt{4ks^2+8s^2+1}}{2}+\frac{1}{2}\right)\\
 			&=(14s^3+7ks^3)\sqrt{4ks^2+8s^2+1}-(34ks^4+2ks^2+4s^2+40s^4+7k^2s^4)\\
 			&=-\frac{(49k^4+280k^3+540k^2+368k+32)s^8+(28k^3+143k^2+236k+124)s^6+(4k^2+16k+16)s^4}{(14s^3+7ks^3)\sqrt{4ks^2+8s^2+1}+(34ks^4+2ks^2+4s^2+40s^4+7k^2s^4)}<0.
 		\end{split}
 	\end{equation*}
 	\item 	\begin{equation*}
 		\begin{cases}
 			h_3\left( \left( \frac{3}{2}k+\frac{7}{2}\right)s+\frac{1}{2}\right)=\frac{1}{4}\left( k^4+2k^3-3k^2-4k+4 \right)s^4-\frac{1}{4}\left(2k^2+10k+13\right)s^2+\frac{1}{4},\\
 			h'_3\left( \left( \frac{3}{2}k+\frac{7}{2}\right)s+\frac{1}{2}\right)=(2k^3+9k^2+19k+20)s^3-(2k+7)s>0,\\
 			h''_3\left( \left( \frac{3}{2}k+\frac{7}{2}\right)s+\frac{1}{2}\right)=2(6k^2+30k+45)s^2-4>0.
 		\end{cases}
 	\end{equation*}
 	\end{itemize} 	
  
  	By observing the above systems, we get the following relations:
 	\begin{table}[htbp]
 		\centering
 		\small
 		\begin{tabular}{|c|c|c|}
 			\hline
 			\multirow{2}*{$n\in\left[(k+2)s,\left( k+\frac{11}{4}\right)s+\frac{1}{2}  \right)$}   & $\Phi(x_{2,1},s)>0$  &  \multirow{2}*{$x_{2,1}>\rho_f>x_{2,2}$}\\
 			 \cline{2-2}& $x_{2,1}> n-(k+2)s-1\geq\xi_f$ & \\
 			 \hline
 			 \multirow{3}*{$n\in\left[\left( k+\frac{11}{4}\right)s+\frac{1}{2}, r_2  \right)$} & $h_3(n)>0$  &  \multirow{3}*{$x_{2,1}>\rho_f>x_{2,2}$}\\
 			 \cline{2-2}& $\Phi(x_{2,1},s)>0$ & \\
 			 \cline{2-2} & $x_{2,1}> n-(k+2)s-1\geq\xi_f$ & \\
 			 \hline
 			 \multirow{3}*{$n\in\left[r_2,   ks+2s+\frac{\sqrt{4ks^2+8s^2+1}}{2}+\frac{1}{2}\right)$} & $h_3(n)\leq0$ & \multirow{3}*{$\rho_f\geq x_{2,1}>\xi_f$}\\
 			 \cline{2-2}& $\Phi(x_{2,1},s)\leq0$ & \\
 			 \cline{2-2} & $x_{2,1}> n-(k+2)s-1\geq\xi_f$ & \\
 			 \hline
 			 \scriptsize{$n\in\left[ks+2s+\frac{\sqrt{4ks^2+8s^2+1}}{2}+\frac{1}{2},(\frac{3}{2}k+\frac{7}{2})s+\frac{1}{2}+\frac{1}{8s}\right] $} & $n-(k+2)s-1\geq x_{2,1}$ & $\rho_f>x_{2,1}$\\
 			 \hline
 		 \end{tabular}
 	\end{table}
	
	\begin{figure}[htbp]
		\centering 
		\includegraphics[width=10cm]{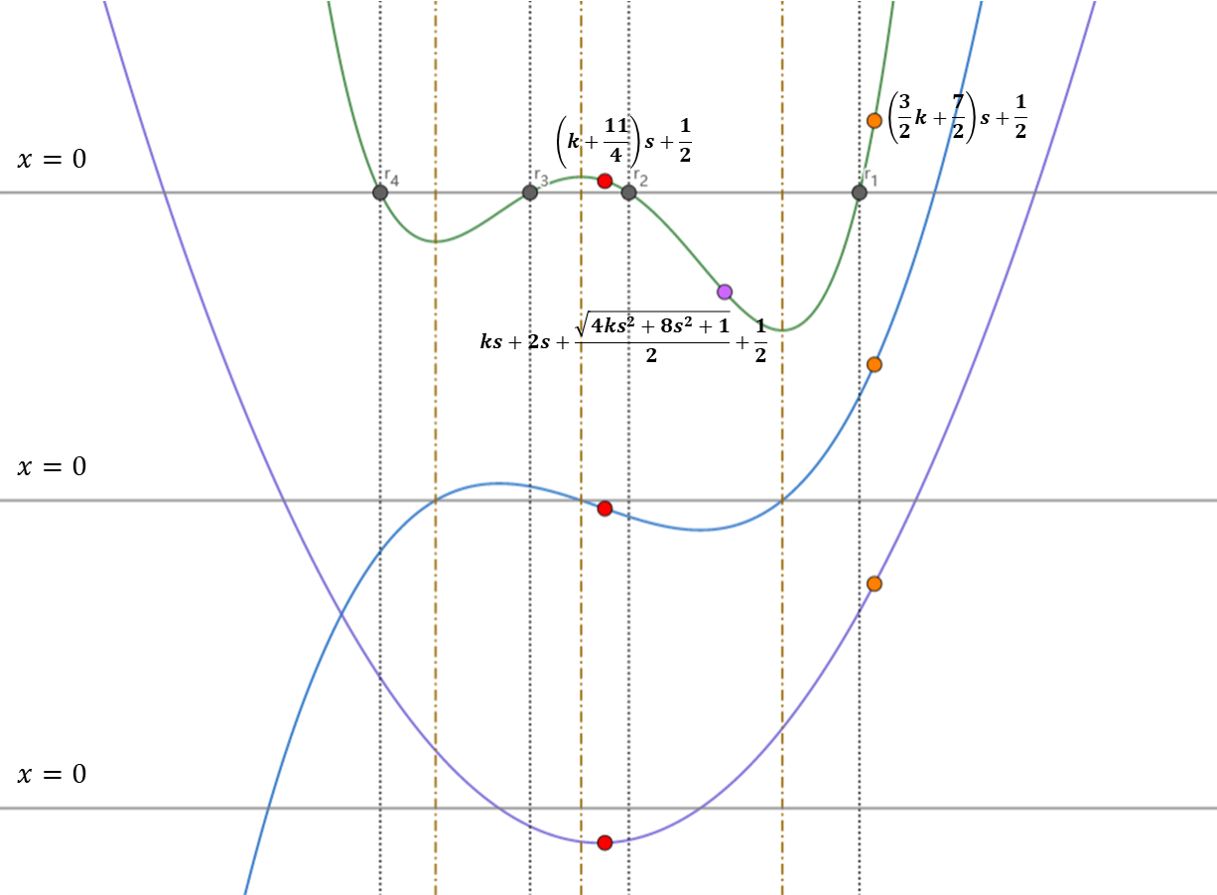}
		\caption{Three functions $h_3(n)$, $h'_3(n)$ and $h''_3(n)$.}
		\label{fig}
	\end{figure}
	
	Let $n_0=r_2$. Then 
	\begin{equation}
		\begin{cases}
			g_2(\rho_f)<0, & \mbox{if } n\in \left[(k+2)s, n_0 \right), \\
			g_2(\rho_f)\geq0, & \mbox{if } n\in \left[n_0,  (\frac{3}{2}k+\frac{7}{2})s+\frac{1}{2}+\frac{1}{8s}\right].
		\end{cases}
	\end{equation}
	
	This completes the proof of Claim 2.
	\end{proof}

	Combining the proof of Claim 2, one can see that there exists $s_0=s_0(n_0)\in [1,\lfloor\frac{n}{k+2}\rfloor]$ such that 
	\begin{equation}
		\begin{cases}
			\frac{\partial \Phi(x,s)}{\partial s}>0, & \mbox{if } x=\rho_f \mbox{ and }s\in \left[1, s_0 \right), \\
			\frac{\partial \Phi(x,s)}{\partial s}\leq0, & \mbox{if } x=\rho_f \mbox{ and } s\in \left[s_0, \lfloor\frac{n}{k+2}\rfloor\right].
		\end{cases}
	\end{equation}

	It can be checked that, by the Implicit–Function Theorem, $x$ is determined by $s$ from equation $\Phi(x, s) = 0$ in $D=\{(x,s)|x=\rho_f,1\leq s\leq \lfloor\frac{n}{k+2}\rfloor\}$. Then we get
	\begin{equation*}
		\frac{\partial \Phi(x,s)}{\partial x}\frac{\partial x(s)}{\partial s}+\frac{\partial \Phi(x,s)}{\partial s}=0.
	\end{equation*}
	
	Based on the above discussion, it follows that
	\begin{equation}\label{sys1}
		\begin{cases}
			\frac{\partial x(s)}{\partial s}=-\frac{\frac{\partial \Phi(x,s)}{\partial s}}{\frac{\partial \Phi(x,s)}{\partial x}}<0, & \mbox{if } x=\rho_f \mbox{ and } s\in \left[1, s_0 \right), \\
			\frac{\partial x(s)}{\partial s}=-\frac{\frac{\partial \Phi(x,s)}{\partial s}}{\frac{\partial \Phi(x,s)}{\partial x}}\geq 0, & \mbox{if } x=\rho_f \mbox{ and } s\in \left[s_0, \lfloor\frac{n}{k+2}\rfloor\right].
		\end{cases}
	\end{equation}
	
	Thus, we have
	$\lambda_{1}(K_s\vee \left( K_{n-(k+2)s}\cup(k+1)sK_1\right))\leq\max\left\lbrace \lambda_{1}(K_t\vee \left( K_{n-(k+2)t}\cup(k+1)tK_1\right)),\right.\\
	\left.\lambda_{1}\left( K_{\lfloor\frac{n}{k+2}\rfloor}\vee \left( K_{n-(k+2)\lfloor\frac{n}{k+2}\rfloor}\cup(k+1)\lfloor\frac{n}{k+2}\rfloor K_1\right)\right) \right\rbrace $ for $t\leq s\leq \lfloor\frac{n}{k+2}\rfloor$. 
	
	This completes the proof.
	\end{proof}
	
	By Lemma \ref{lem4_1} and Lemma \ref{lemma4_2}, we deduce the following theorem directly.
	\begin{theorem}\label{leaf_k}
		Let $G$ be an $n$-vertex $t$-connected graph. If
		\begin{equation*}
			\lambda_1(G)\geq\max\left\lbrace \lambda_{1}(K_t\vee \left( K_{n-(k+2)t}\cup(k+1)tK_1\right)),\lambda_{1}\left( K_{\lfloor\frac{n}{k+2}\rfloor}\vee \left( K_{n-(k+2)\lfloor\frac{n}{k+2}\rfloor}\cup(k+1)\lfloor\frac{n}{k+2}\rfloor K_1\right)\right) \right\rbrace,
		\end{equation*} then $G$ has a spanning tree with leaf degree at most $k$ unless $G\cong K_t\vee(K_{n-(k+2)t}\cup(k+1)tK_1)$ or $G\cong K_{\lfloor\frac{n}{k+2}\rfloor}\vee \left( K_{n-(k+2)\lfloor\frac{n}{k+2}\rfloor}\cup(k+1)\lfloor\frac{n}{k+2}\rfloor K_1\right)$.
	\end{theorem}

If $t=1$, by Theorem \ref{leaf_k}, we can get the following result. This generalizes a result of Theorem 1.5 in \cite{ao}.
\begin{corollary}\label{cor}
	Let $G$ be a connected graph of order $n\geq 2k+12$, where $k\geq1$ is an integer. If $\lambda_1(G)\geq \lambda_1(K_1\vee \left( K_{n-(k+2)}\cup(k+1)K_1\right))$, then $G$ has a spanning tree with leaf degree at most $k$ unless $G\cong K_1\vee(K_{n-(k+2)}\cup(k+1)K_1)$.
\end{corollary}

\begin{remark}
	Applying the technique in Lemma \ref{lemma4_2}, one can further study the relationships between the spectral radius ($A_\alpha$-spectral radius) and the existence of other subgraphs with constrained degree (such as $[1,b]$-odd factors, regular factors, $k$-trees, etc.). Furthermore, inspired by \eqref{sys1}, we can explore similar problems in graphs that possess certain special graph parameters.
\end{remark}

	\section{Two spectral conditions for the existence of a spanning tree with leaf degree at most $k$ in graphs}
	In this section, we present some sufficient conditions for a graph to have a spanning tree with leaf degree at most $k$.
	\begin{lemma}\cite{gu2}\label{k1101}
		Let $G$ be a graph with $n$ vertices and at least one edge. Suppose that $S\subset V(G)$ such that $G-S$ is disconnected. Let $X$  and $Y$ be disjoint vertex subsets of $G-S$ such that $X\cup Y=V(G)-S$ with $|X|\leq|Y|$. Then $$|X|\leq\frac{\mu_1-\mu_{n-1}}{2\mu_1}{n},$$  and
		\begin{equation}\label{eq20}
			|S|\geq\frac{2\mu_{n-1}}{\mu_1-\mu_{n-1}}|X|,
		\end{equation}
		with each equality holding only when $|X|=|Y|$.
	\end{lemma}
	

	Based on Lemma \ref{k1101}, we can give a same sufficient condition for a graph to have a spanning tree with leaf degree at most $k$ in terms of the Laplacian eigenvalues as below.
	\begin{theorem}\label{k1210}
		Let  $G$ be  a graph with the Laplacian eigenvalues $\mu_1\geq\mu_2\geq\cdots\geq\mu_{n-1}\geq\mu_n=0$. If $\mu_1<(k+1)\mu_{n-1}$, then $G$ has a spanning tree with leaf degree at most $k$.
	\end{theorem}
	
	\begin{proof}
		Assume to the contrary that $G$ has no spanning trees with leaf degree at most $k$. Then by Theorem \ref{leafk}, we get $i(G-S)\geq (k+1)|S|$. Let $G_1,G_2,\dots,G_c$ be components of $G-S$ with $|V(G_1)|\leq|V(G_2)|\leq\dots\leq|V(G_c)|$, where $c=c(G-S)$. Clearly, $c = c(G-S)\geq i(G-S)=i$. If $c>i$, let $X = \cup_{1\leq j\leq \lceil\frac{i+1}{2}\rceil}V(G_j)$ and $Y=V(G)-S-X$. Then $\frac{i+1}{2}\leq|X|\leq |Y|$. By \eqref{eq20} and $(k+1)\mu_{n-1}> \mu_1$,
		\begin{equation*}
			|S|\geq\frac{2\mu_{n-1}}{\mu_1-\mu_{n-1}}|X|\geq\frac{(i+1)\mu_{n-1}}{\mu_1-\mu_{n-1}}>\frac{i+1}{k}.
		\end{equation*}
		Hence $i(G-S)< (k+1)|S|-|S|-1$, a contradiction. If $c=i$ and $i$ is even, we let $X=\cup_{1\leq j\leq \frac{i}{2}}V(G_j)$ and $Y=V(G)-S-X$. Then $|X|=|Y|=\frac{i}{2}$. By \eqref{eq20} and $(k+1)\mu_{n-1}> \mu_1$,
		\begin{equation*}
			|S|=\frac{2\mu_{n-1}}{\mu_1-\mu_{n-1}}|X|=\frac{i\cdot\mu_{n-1}}{\mu_1-\mu_{n-1}}>\frac{i}{k}.
		\end{equation*}
		Thus, $i(G-S)< (k+1)|S|-|S|$, a contradiction. If $c=i$ and $i$ is odd, we let $X=\cup_{1\leq j\leq \frac{i-1}{2}}V(G_j)$ and $Y=V(G)-S-X$. Then $|X|=\frac{i-1}{2}<|Y|=\frac{i+1}{2}$. By \eqref{eq20} and $(k+1)\mu_{n-1}> \mu_1$,
		\begin{equation*}
			|S|>\frac{2\mu_{n-1}}{\mu_1-\mu_{n-1}}|X|=\frac{(i-1)\mu_{n-1}}{\mu_1-\mu_{n-1}}>\frac{i-1}{k}.
		\end{equation*}
		Therefore, $i(G-S)< (k+1)|S|-|S|+1$, i.e., $i(G-S)<(k+1)|S|$, a contradiction.
		
		This completes the proof.
	\end{proof}
	
	\begin{corollary}
		Let  $G$ be a $t$-connected graph with the Laplacian eigenvalues $\mu_1\geq\mu_2\geq\cdots\geq\mu_{n-1}\geq\mu_n=0$, where $t\geq2$. If $\mu_1\leq(k+1)\mu_{n-1}$, then $G$ has a spanning tree with leaf degree at most $k$.
	\end{corollary}

	The complement $\overline{G}$ of $G$ is the  graph whose vertex set is $V(G)$ and whose edges are the pairs of nonadjacent vertices of $G$. We now give sufficient conditions for the existence of a spanning tree with leaf degree at most $k$ in a graph in terms of the spectral radius of its complement.
	\begin{theorem}\label{t0403}
		Let  $G$ be  an $n$-vertex connected graph with minimum degree $\delta$ and let $\overline{G}$ be the complement of $G$. If $\lambda_1(\overline{G})<(k+1)\delta-1$, then $G$ has a spanning tree with leaf degree at most $k$.
	\end{theorem}
	
	\begin{proof}
		Assume to the contrary that $G$ has no spanning trees with leaf degree at most $k$. By Theorem \ref{leafk}, there exists some nonempty subset $S\subseteq V(G)$ such that $i(G-S)\geq(k+1)|S|$. Let $A$ denote the set of isolated vertices in $G-S$. Since the neighbors of each isolated vertex belong to $S$, we have $|S|\geq\delta$, which means that $|A|\geq(k+1)|S|\geq(k+1)\delta$. Since $\overline{G}[A]$ is a clique, we have  $\lambda_1(\overline{G})\geq\lambda_1(\overline{G}[A])=|A|-1\geq(k+1)\delta-1$ by Lemma \ref{radius}, a contradiction.
	\end{proof}

	\section*{Declaration of competing interest}
	There is no conflict of interest.
	
	\section*{Acknowledgement(s)}
	This work was supported by the National Natural Science Foundation of China (Grant No. 61773020). The authors would like to express their sincere gratitude to all the referees for their careful reading and insightful suggestions.


\begin{thebibliography}{99}	
%
%
%
%
%
%
%
%
%
%
%
%
%
%
%
%
%
%
%
%
%
%
%
%
%
%
%
%
%
%
%
%
%
%
%
%
%
%
%
%
%
%
%
%
%
%
%
%
%
%
%
%
%
%
%
%
%
%
%
%
%
%
%
	\bibitem{ao} G. Ao, R. Liu, J. Yuan, Spectral radius and spanning trees of graphs, Discrete Math. 346 (2023) 113400.

	\bibitem{boll} B. Bollob\'as, J. Lee, S. Letzter, Eigenvalues of subgraphs of the cube, European J. Combin. 70 (2018) 125-148.
	
	\bibitem{brou} A.E. Brouwer, W.H. Haemers, Eigenvalues and perfect matchings, Linear Algebra Appl. 395 (2005) 155–162.
	
	\bibitem{brou2} A.E. Brouwer, W.H. Haemers, Spectra of graphs, Springer, New York, 2011.
	
	\bibitem{cioa} S.M. Cioaba, D.A. Gregory, W.H. Haemers, Matchings in regular graphs from eigenvalue, J. Combin. Theory Ser. B. 99 (2009) 287–297.
	
	\bibitem{duan} C. Duan, L. Wang, X. Liu, Edge connectivity, packing spanning trees, and eigenvalues of graphs, Linear Multilinear Algebra 68 (2018) 1077-1095.
	
	\bibitem{ell} M.N. Ellingham, X. Zha, Toughness, trees, and walks, J. Graph Theory 33 (2000) 125–137.
	
	
	\bibitem{fan} D. Fan, S.V. Goryainov, X. Huang, H. Lin, The spanning $k$-trees, perfect matchings and spectral radius of graphs, Linear and Multilinear Algebra 70(21) (2022) 7264-7275.
	
	
	\bibitem{gu1} X. Gu, Regular factors and eigenvalues of regular graphs, European J. Combin. 42 (2014) 15–25.
	
	\bibitem{gu2} X. Gu, M. Liu, A tight lower bound on the matching number of graphs via laplacian eigenvalues, European J. Combin. 101 (2022) 103468.
	
	
	
	
	\bibitem{gods} C. Godsil, G. Royle, Algebraic graph theory, Graduate Texts in Mathematics, 207, Springer‐Verlag, New	York, 2001.
	
	\bibitem{hwang} S.G. Hwang, Cauchy's interlace theorem for eigenvalues of Hermitian matrices, Am. Math. Mon. 111 (2004) 157–159.
	
	\bibitem{kan} A. Kaneko, Spanning trees with constraints on the leaf degree, Discrete Appl. Math. 115 (2001) 73–76.
	
	\bibitem{kim3}  M. Kim, S. O, W. Sim, D. Shin, Matchings in graphs from the spectral radius, Linear and Multilinear Algebra 71(11) (2023) 1794–1803.
	
	\bibitem{kim1} D. Kim, S. O, Eigenvalues and parity factors in graphs with given minimum degree, Discrete Math. 346 (2023) 113290.
	
	\bibitem{kim2} S. Kim, S. O, J. Park, H. Ree, An odd $[1,b]$-factor in regular graphs from eigenvalues, Discrete Math. 343 (2022) 111906.
	
	\bibitem{lu1} H. Lu, Regular factors of regular graphs from eigenvalues, Electron. J. Combin. 17 (2010) \#R159.
	
	\bibitem{lu2} H. Lu, Regular graphs, eigenvalues and regular factors, J. Graph Theory 69 (2012) 349–355.
	
	\bibitem{lu3} H. Lu, Z. Wu, X. Yang, Eigenvalues and $[1, n]$-odd factors, Linear Algebra Appl. 433 (2010) 750–757.
	
	\bibitem{o1} S. O, Eigenvalues and $[a,b]$-factors in regular graphs, J. Graph Theory 100 (2022) 458-469.
	
	
	\bibitem{o2} S. O, Spectral radius and matchings in graphs, Linear Algebra Appl. 614 (2021) 316-324
	
	
	\bibitem{win}  S. Win, On a connection between the existence of $k$-trees and the toughness of a graph, Graphs Combin. 5 (1989) 201–205.
	
	\bibitem{zhang} W. Zhang, The maximum spectral radius of $t$-connected graphs with bounded matching number, Discrete Math. 345 (2022) 112775.
	
	
	
	\end{thebibliography}
\end{document}